\documentclass{article} 
\usepackage{amssymb,amsmath,amsthm}
\usepackage{amscd}
\usepackage{hyperref}
\usepackage{epsfig}
\usepackage{amsfonts}
\usepackage{amssymb}
\usepackage{amsmath,enumerate}
\usepackage{commath}
\usepackage{euscript}
\usepackage{enumitem}
\usepackage{amscd}
\usepackage{xcolor}
\usepackage[ruled,vlined]{algorithm2e}
\usepackage[numbers,sort&compress]{natbib}

\newtheorem{theorem}{Theorem}[section]

\newtheorem{lemma}{Lemma}[section]

\newtheorem{example}{Example}[section]

\makeatletter
\newsavebox{\@brx}
\newcommand{\llangle}[1][]{\savebox{\@brx}{\(\m@th{#1\langle}\)}%
	\mathopen{\copy\@brx\kern-0.5\wd\@brx\usebox{\@brx}}}
\newcommand{\rrangle}[1][]{\savebox{\@brx}{\(\m@th{#1\rangle}\)}%
	\mathclose{\copy\@brx\kern-0.5\wd\@brx\usebox{\@brx}}}
\makeatother
\begin{document}
	\title{ Mizoguchi-Takahashi local contractions to Feng-Liu contractions}
	\author{
		Pallab Maiti\footnotemark[2], Asrifa Sultana\footnotemark[1] \footnotemark[2]}
	\date{ }
	\maketitle
	\def\thefootnote{\fnsymbol{footnote}}
	
	\footnotetext[1]{ Corresponding author. e-mail- {\tt asrifa@iitbhilai.ac.in}}
	\noindent
	\footnotetext[2]{Department of Mathematics, Indian Institute of Technology Bhilai, Raipur - 492015, India.
	}
	
	
		
	\begin{abstract}
	In this article, we establish that any uniformly local Mizoguchi-Takahashi contraction is actually a set-valued contraction due to Feng and Liu on a metrically convex complete metric space. Through an example, we demonstrate that this result need not hold on any arbitrary metric space. Furthermore, when the metric space is compact, we derive that any Mizoguchi-Takahashi local contraction and Nadler local contraction are equivalent. Moreover, a result related to invariant best approximation is established.
	\end{abstract}
	
\section{Introduction}\label{sec1}

Let $(Z,d)$ be a metric space and $CL(Z)$ denotes the collection of non-void closed subsets of $Z$. The set $CB(Z)$ contains all the non-void bounded closed subsets of $Z$. For any $P,Q\in CB(Z)$, 
\begin{equation}
	H(P,Q)=\max\bigg\{\sup_{u\in Q}D(u,P), \sup_{v\in P}D(v,Q)\bigg\}
\end{equation}
is called Hausdorff metric on $CB(Z)$, where $D(u,P)=\inf\{d(u,v):v\in P\}$. For a mapping $F:Z\to CL(Z)$, an element $z^*\in Z$ is known as a fixed point for the map $F$ if $z^*\in Fz^*$. The occurrence of fixed points for set-valued mappings was first deduced by Nadler \cite{nadler} in the year of 1969. Later, Mizoguchi and Takahashi \cite{mizo} improved the Nadler's \cite{nadler} theorem for a set-valued generalized contraction in the year 1989. The authors \cite{mizo} established the below stated theorem.
\begin{theorem}\cite{mizo}\label{mizo_thm}
	Suppose that $(Z,d)$ is complete and a set-valued map $F:Z\to CB(Z)$ in order that for every $y,z\in Z$,
	\begin{equation}\label{mizo_eq}
		H(Fy,Fz)\leq k(d(y,z))d(y,z),
	\end{equation}
	where $k\in W=\{\theta:[0,\infty)\to [0,1)| \limsup_{t\to s^+}\theta(t)<1,~\forall s\in [0,\infty)\}$. Then there is an element $z^*\in Fz^*$. 
\end{theorem} 
In fact, Suzuki \cite{suzuki} provided an example to deduce that every Mizoguchi -Takahashi contraction \cite{mizo} is not necessarily a Nadler contraction \cite{nadler} on arbitrary metric spaces in the year of 2007. However, Eldred-Anuradha-Veeramani \cite{eldred} observed that any Mizoguchi-Takahashi \cite{mizo} contraction follows a Nadler \cite{nadler} contraction, whenever the metric space $(Z,d)$ is metrically convex complete \cite{eldred} in the year 2009. There are several extensions of the Theorem \ref{mizo_thm} due to Mizoguchi and Takahashi \cite{mizo} in the literature, which can be found in \cite{berinde,mam}.

In the year 1961, Edelstein \cite{edel} generalized the famous Banach contraction principle for uniformly local contraction \cite{edel}. Subsequently, Nadler \cite{nadler} improved the Edelstein's \cite{edel} result for set-valued uniformly local contraction \cite{nadler}. Later, Sultana and Vetrivel \cite{mam} extended the Theorem \ref{mizo_thm} due to Mizoguchi and Takahashi \cite{mizo} for set-valued uniformly local contraction. The authors \cite{mam} derived the succeeding theorem.
\begin{theorem}\cite{mam}\label{mizo_local_thm}
	Suppose that $(Z,d)$ is complete $r$-chainable and a set-valued map $F:Z\to CB(Z)$ in order that for each $y,z\in Z$ having $d(y,z)<r$ (where $r>0$) fulfils,
	\begin{equation}\label{mizo_local}
		H(Fy,Fz)\leq k(d(y,z))d(y,z), 
	\end{equation}
	where $k\in W$. Then there is an element $z^*\in Fz^*$.
\end{theorem}  

It is worth to note that every Mizoguchi-Takahashi local contraction \cite{mam} is not a Nadler \cite{nadler} local contraction, see \cite[Example 3.1]{mam_equivalence}. But, Sultana and Qin \cite{mam_equivalence} demonstrated that these two contractions are equivalent, whenever the metric space $(Z,d)$ is metrically convex complete in the year 2019.

On the other hand, Feng and Liu \cite{feng} improved Nadler's \cite{nadler} theorem for the mappings $F$ from $Z$ into $CL(Z)$ in the year of 2006, which is stated below.
\begin{theorem}\cite{feng}\label{feng_thm}
	Suppose that $(Z,d)$ is a complete metric space and a set-valued map $F:Z\to CL(Z)$ in order that for each $y\in Z$, there is $z\in I_{c\in (0,1)}^y=\{x\in Fy:cd(x,y)\leq D(y,Fy)\}$,
	\begin{equation}\label{feng_eq}
		D(z,Fz)\leq \alpha d(y,z)~\textrm{where}~ 0\leq \alpha\leq c<1.
	\end{equation}
	Then there an element $z^*\in Fz^*$ if $h:Z\to \mathbb{R}$ by $h(z)=D(z,Fz)$ is lower semicontinuous.
\end{theorem}

In this manuscript, inspired by Eldred et. al \cite{eldred} and Sultana-Qin \cite{mam_equivalence}, we establish that any set-valued contraction due to Feng-Liu \cite{feng} follows from the Mizoguchi-Takahashi local contraction \cite{mam_equivalence} on a metrically convex metric space. On any arbitrary metric spaces, this result need not be true, an example is provided for this purpose. The fixed points for the set-valued mappings fulfil the equation $(\ref{mizo_local})$ are also studied on metrically convex spaces. Furthermore, the equivalence of Mizoguchi-Takahashi local contraction \cite{mam_equivalence} and Nadler local contraction \cite{nadler} is established in a compact metric space. An invariant approximation result is also discussed through our main theorem.
\section{Notations and Definitions}
This segment contains with some mathematical definitions and notations, which are needful up to the end of this paper. For the metric space $(Z,d)$, let $u,v$ be any two points in $Z$. Then a point $w$ in $Z$, where $u\neq w\neq v$, is called metrically \cite{eldred} between $u$ and $v$ if $d(u,v)=d(u,w)+d(w,v)$. The space $(Z,d)$ is called metrically convex \cite{eldred} if for every pair of elements $u$ and $v$ in $Z$, there is an element $w\in Z$ in order that $w$ is metrically between $u$ and $v$. The below stated lemma is regarding a property of metrically convex space, which is needful for establish our main theorem.

\begin{lemma}\label{metrically}\cite{eldred}
	Assume that $(Z,d)$ is metrically convex complete and $u,v\in Z$ are any arbitrary points. Then there is $[a_1,a_2]\subset \mathbb{R}$ and an isomorphism $\varphi:[a_1,a_2]\to Z$ in order that $\varphi(a_1)=u$, $\varphi(a_2)=v$.
\end{lemma} 

Again for the metric space $(Z,d)$, consider a self map $\Psi$ on $Z$. Then $\Psi$ is defined as uniformly local contractions \cite{edel} if for each $y,z \in Z$ having $d(y,z)<r$ fulfils $d(\Psi (y),\Psi (z))\leq \alpha d(y,z)$, where $0\leq \alpha <1$ and $r>0$. Furthermore, for $y,z\in Z$, a $r$-chain \cite{edel} between $y$ and $z$ means, there is a finite sequence $y=z_0,z_1,\cdots,z_n=z$ in $Z$, such that for every $j\in\{1,2,\cdots,n\}$, $d(z_{j-1},z_j)<r$. The space $(Z,d)$ is stated as $r$-chainable \cite{edel} whenever there is a $r$-chain between every points $y$ and $z$ in $Z$.

%
%
\section{Main Result}
Now we commence our main theorem about that any Mizoguchi-Takahashi local contraction \cite{mam}, that is, the mappings that meet the equation (\ref{mizo_local}), is actually a Feng-Liu contraction \cite{feng} on a metrically convex metric space $(Z,d)$.
\begin{theorem}\label{main}
	Suppose that $(Z,d)$ is a complete metric space having metrically convexity property and a set-valued map $F:Z\to CB(Z)$ in order that for $y,z\in Z$ having $d(y,z)<r$ (where $r>0$),
	\begin{equation}\label{MT}
		H(Fy,Fz)\leq k(d(y,z))d(y,z)
	\end{equation}
	where $k\in W$. Then $F$ becomes a Feng-Liu contraction, that is, for each $y^{'}\in Z$, there is $z^{'}\in I_c^{y^{'}}$ fulfils,
	\begin{equation*}
		D(z^{'},Fz^{'})\leq \alpha d(y^{'},z^{'})~\textrm{where}~0\leq \alpha <c<1.
	\end{equation*}
\end{theorem}
\begin{proof}
	Our main goal is to show that there is certain $\alpha$, where $0\leq \alpha<c<1$ in order that for $y\in Z$, $z\in I_c^y$ meets $D(z,Fz)\leq \alpha d(y,z)$. Suppose that
	\begin{equation}
		M=\{m>0: \sup\{k(d(y,z)):d(y,z)\in [0,m]\}=1\}.
	\end{equation}
	Then there are two possibilities that could happen. Firstly, we assume that the set $M$ is void. Then for each positive $m$, the supremum value of $k(d(y,z))$ is strictly less than $1$, where $d(y,z)\in [0,m]$.
	Since $r>0$ is fixed, we assume that 
	\begin{equation}\label{set}
		\sup\{k(d(y,z)):d(y,z)\in [0,r]\}=\alpha<1.
	\end{equation}
	Evidently, $\alpha\in [0,1)$. Let $y^{'}\in Z$ be a fixed element. Then there is an element $z^{'}\in I_c^y{^{'}}$ where $c>\alpha$, due to the fact that $F(y^{'})\subset CL(Z)$. Consequently, for chosen $y^{'}$ and $z^{'}$, an isometry $\varphi_1:[t_1,t_2]\to Z$ exists satisfies $\varphi_1(t_1)=y^{'}$ and $\varphi_1(t_2)=z^{'}$ (where $[t_1,t_2]\subset \mathbb{R}$ ) according to the Lemma \ref{metrically}.
	Now for a fixed positive real number $q<r$, we can always find $L\in \mathbb{N}$ satisfying $t_1+Lq<t_2$ and $t_2\leq t_1+(L+1)q$. Consequently,
	\begin{equation}\label{isomorphism}
		\begin{aligned}
			d(y^{'},z^{'})&=d(\varphi_1(t_1),\varphi_1(t_2))=|t_1-t_2|\nonumber\\
			&=|t_1-(t_1+q)|+|(t_1+q)-(t_1+2q)|+\cdots+|(t_1+Lq)-t_2|\nonumber\\
			&=d(y^{'},\varphi_1(t_1+q))+d(\varphi_1(t_1+q),\varphi_1(t_1+2q))+\cdots+\nonumber\\
			&\qquad d(\varphi_1(t_1+Lq),z^{'}).
		\end{aligned}
	\end{equation}
	Again, it is simple to observe that $d(z^{'},\varphi_1(t_1+Lq))\leq q<r$ and $d(\varphi_1(t_1+nq),\varphi_1(t_1+(n+1)q))=q<r$ for every $0\leq n<L$, that is every terms of right hand side of the equation (\ref{isomorphism}) is less than $r$. Therefore for chosen $y^{'}\in Z$ and $z^{'}\in I_c^{y^{'}}$, it yields
		\begin{align*}
			D(z^{'},Fz^{'})&\leq H(Fy^{'},Fz^{'})\\
			&\leq H(Fy^{'},F(\varphi_1(t_1+q)))+\cdots+H(F(\varphi_1(t_1+Lq)),Fz^{'})\\
			&\leq  \alpha[d(y^{'},\varphi_1(t_1+q))+\cdots+d(\varphi_1(t_1+Lq),z^{'})]\\
			&\leq \alpha d(y^{'},z^{'}).\qquad[\text{using}~(\ref{isomorphism}) ]
		\end{align*}

	Next, we assume that the set $M$ is non-void and $m_0$ is the infimum of $M$. If possible, let $m_0=0$. Then we are able to find two sequences $\{d(y_n,z_n)\}_n$, $\{k(d(y_n,z_n))\}_n$ in order that $d(y_n,z_n)\to 0$ and $k(d(y_n,z_n))\to 1$, which contradicts to the criteria of the map $k\in W$. Hence $m_0\neq 0$. Then for a chosen positive real $q_0<m_0$, we have
	\begin{equation}\label{set1}
		\sup\{k(d(x_1,x_2)):d(x_1,x_2)\in [0,q_0]\}=\alpha_0<1.
	\end{equation} 
	Let $y^{'}\in Z$ be a fixed element. Subsequently, there is $z^{'}\in I_c^{y^{'}}$, where $c>\alpha_0$. Then for that $y^{'}$ and $z^{'}$, there is an isometry $\varphi_2:[s_1,s_2]\to Z$ in order that $\varphi_2(s_1)=y^{'}$ and $\varphi_2(s_2)=z^{'}$, where $[s_1,s_2]\subset \mathbb{R}$. Consequently, for certain fixed positive number $a<\min\{q_0,r\}$, we can find $L^{'}\in \mathbb{N}$ satisfying $s_1+L^{'}a<s_2$ and $s_2\leq s_1+ (L^{'}+1)a$.
	Evidently,
	\begin{equation}\label{isomorphism1}
		\begin{aligned}
			d(y^{'},z^{'})&=d(\varphi_2(s_1),\varphi_2(s_2))=|s_1-s_2|\nonumber\\
			&=|s_1-(s_1+a)|+|(s_1+a)-(s_1+2a)|+\cdots+|(s_1+L^{'}a)-s_2|\nonumber\\
			&=d(y^{'},\varphi_2(s_1+a))+d(\varphi_2(s_1+a),\varphi_2(s_1+2a))+\cdots+\nonumber\\
			&\qquad d(\varphi_2(s_1+L^{'}a),z^{'}).
		\end{aligned}
	\end{equation}
	Again, it is simple to check that $d(z^{'},\varphi_2(s_1+L^{'}a))\leq a$ and $d(\varphi_2(s_1+na),\varphi_2(s_1+(n+1)a))=a$ for every $0\leq n<L^{'}$, that is each terms of right hand side of the equation (\ref{isomorphism1}) is less than or equal to $a$. Therefore for that chosen $y^{'}\in Z$ and $z^{'}\in I_c^{y^{'}}$, it yields	
	\begin{align*}
		D(z^{'},Fz^{'})\leq& H(Fy^{'},Fz^{'})\\
		\leq&H(Fy^{'},F(\varphi_2(s_1+a)))+\cdots+H(F(\varphi_2(s_1+L^{'}a)),Fz^{'})\\
		\leq & \alpha_0[d(y^{'},\varphi_2(s_1+a))+\cdots+d(\varphi_2(s_1+L^{'}a),z^{'})]\\
		\leq&\alpha_0 d(y^{'},z^{'})\qquad[\text{using}~(\ref{isomorphism1})].
	\end{align*}
	Thus $F$ becomes a Feng-Liu contraction. 
\end{proof}

In the succeeding theorem, we establish the fixed points for Mizoguchi-Takahashi local contraction on a metrically convex space. Indeed, this result follows from our above mentioned Theorem \ref{main} and the Theorem \ref{feng_thm} due to Feng and Liu \cite{feng}. Although, this theorem is established by Sultana and Qin \cite[Corollary 3.1]{mam_equivalence} in the year 2019. 
\begin{theorem}\label{cor1}\cite{mam_equivalence}
	Suppose that $(Z,d)$ is metrically convex complete and a set-valued map $F:Z\to CB(Z)$ meets the equation $(\ref{mizo_local})$. Then presence of $z^*\in Z$ in order that $z^*\in Fz^*$ can be assured.
\end{theorem}
\begin{proof}
	Since $F$ meets the equation $(\ref{mizo_local})$ and $Z$ is metrically convex, then by Theorem \ref{main}, the map $F$ fulfils the Feng-Liu contraction. As a consequence, for each $y\in Z$, $z\in I_c^y$ where $0\leq \alpha<c<1$ fulfils
	\begin{eqnarray*}
		D(z,Fz)\leq \alpha d(y,z).
	\end{eqnarray*}
	Now our aim is to establish that $h:Z\to \mathbb{R}$ by $h(z)=D(z,Fz)$ is lower semicontinuous. Let us assume $\{q_n\}_n\in Z$ in order that $q_n\to q$. Then we achieve $l\in \mathbb{N}$ in order that $d(q_n,q)<r$, for each $n\geq l$. Now for every $n\geq 1$ and $z\in Fq_n$,
	\begin{eqnarray}
		D(q,Fq)&\leq&d(q,q_n)+d(q_n,z)+D(z,Fq)\nonumber\\
		&\leq&d(q,q_n)+d(q_n,z)+H(Fq_n,Fq)\nonumber\\
		&\leq&d(q,q_n)+D(q_n,Fq_n)+H(Fq_n,Fq).
	\end{eqnarray}
	As for all $n\geq l$, $d(q_n,q)<r$, then the last inequation leads to 
	\begin{eqnarray*}
		D(q,Fq)&\leq& d(q,q_n)+D(q_n,Fq_n)+k(d(q_n,q))d(q_n,q)\\
		&< & d(q,q_n)+D(q_n,Fq_n)+d(q_n,q).
	\end{eqnarray*}
	Taking $n\to \infty$, we conclude that $D(q,Fq)\leq\liminf_{n\to \infty}D(q_n,Fq_n)$. Now applying the Theorem \ref{feng_thm}, we are able to find $z^*\in Z$ in order that $z^*\in Fz^*$.
\end{proof}

In 2009, Eldred et. al \cite[Theorem 2.4]{eldred} also established that Mizoguchi-Takahashi contraction \cite{mizo} follows Nadler's contraction \cite{nadler}, whenever the metric space is compact. In the below stated theorem, we establish the equivalence between the Mizoguchi-Takahashi local contraction \cite{mam} and the Nadler local contraction \cite{nadler} in a metric space which is compact.
\begin{theorem}
	Let $(Z,d)$ be compact and a set-valued map $F:Z\to CB(Z)$ meets the equation $(\ref{mizo_local})$. Then for each $y,z\in Z$ having $d(y,z)<r$ (where $r>0$), 
	\begin{equation*}
		H(Fy,Fz)\leq \alpha\, d(y,z)\quad \text{where}~~0
		\leq\alpha< 1.
	\end{equation*}
\end{theorem}
\begin{proof}
	Let $0<q<r$ be a fixed real number. Now consider a real number $m$, which is the supremum of $\frac{H(Fy,Fz)}{d(y,z)}$, where the elements $y$ and $z$ lies in $Z$ having $0< d(y,z)\leq q$.
	We want to show that $m<1$. If possible let $m=1$, then there is $\{(y_n,z_n)\}_n\in Z\times Z$ having $d(y_n,z_n)\in (0,q]$ fulfils 
	\begin{equation}\label{limit}
		\lim_{n\to\infty}\frac{H(Fy_n,Fz_n)}{d(y_n,z_n)}=1.
	\end{equation}
	On the account of $d(y_n,z_n)\leq q<r$, we have $\frac{H(Fy_n,Fz_n)}{d(y_n,z_n)}\leq k(d(y_n,z_n))<1$. Taking $n\to \infty$, it yields that $k(d(y_n,z_n))\to 1$. If $d(y_n,z_n)$ converges to $0$, then it contradicts the criteria of $k\in W$. Hence we assert that $\{d(y_n,z_n)\}_n$ converges to a positive real number.
	
	As $\{(y_n,z_n)\}_n\in Z\times Z$ and $Z$ is compact, then there is $\{(y_{n_l},z_{n_l})\}_{l\in \mathbb{N}}$ in order that $(y_{n_l},z_{n_l})\to (y^*,z^*)$. Consequently, $d(y_{n_l},z_{n_l})\to d(y^*,z^*)$. Since $d(y_n,z_n)\in (0,q]$, then $d(y^*,z^*)\in (0,q]$, that is $d(y^*,z^*)\leq q<r$. Now
	\begin{eqnarray*}
		\lim_{l\to\infty}\frac{H(Fy_{n_l},Fz_{n_l})}{d(y_{n_l},z_{n_l})}&=&\frac{H(Fy^*,Fz^*)}{d(y^*,z^*)}\\
		&\leq & \frac{k(d(y^*,z^*))d(y^*,z^*)}{d(y^*,z^*)}\qquad [\because~d(y^*,z^*)<r]\\
		&=&k(d(y^*,z^*))<1.
	\end{eqnarray*}
	This contradicts the equation $(\ref{limit})$. Thus we can conclude that $m<1$, that is there is an non-negative $\alpha$ in order that $m=\alpha<1$. Therefore for each $y,z\in Z$ having $d(y,z)<r$, we have
	\begin{equation*}
		H(Fy,Fz)\leq k(d(y,z))d(y,z)\leq \alpha d(y,z)\quad\text{where}~\alpha\in [0,1).
	\end{equation*}
\end{proof}

Many mathematicians are interested to deal the invariant best approximation problems through fixed point, see (\cite{thagafi,pathak}). In \cite{kamran1,kumam}, the authors established invariant best approximation theorems for set-valued mappings. In the following, using the Theorem \ref{feng_thm} due to Feng-Liu \cite{feng} and our main Theorem \ref{main}, we derive a invariant best approximation theorem for a set-valued mapping which meets the criteria $(\ref{mizo_local})$ on a metrically convex space. Let $P$ be a non-void subset of a normed linear space $(Z,\|.\|)$ and $p\in Z$. Then a collection of elements $y\in P$ in order that $D(p,P)=\|y-p\|$ is described as $B_P(p)$ and it is stated as a best $P$-approximates of $p$ over $P$.

\begin{theorem}
	Let $P$ be a subset of a metrically convex normed linear space $(Z,\|.\|)$ and a set-valued map $F:P\to CB(P)$ in order that for every $y,z\in P$ having $d(y,z)<r$ fulfils $(\ref{mizo_local})$. Suppose that the below stated criteria hold:
	\begin{enumerate}
		\item [(i)] the set $B_P(p)$ is a complete subset of $P$,
		\item [(ii)] for every $x\in B_P(p)$, $\sup_{y\in Fx}\|y-p\| \leq \|x-p\|$.
	\end{enumerate}
	Then there is $z^*\in B_P(p)\cap \{z\in Z: z\in Fz\}$ if the function $h:B_P(p)\to \mathbb{R}$ is lower semicontinuous.
\end{theorem}
\begin{proof}
	As $Z$ is metrically convex and the set-valued map $F:P\to CB(P)$ fulfils $(\ref{mizo_local})$ for each $y,z\in P$ having $d(y,z)<r$, then by the Theorem \ref{main}, for each $y\in B_P(p)\subseteq P$, $z\in I_c^y$ meets $D(z,Fz)\leq \alpha d(y,z)$, where $0\leq \alpha<c<1$. Now let $x\in B_P(p)$ and $z\in Fx$. As $x\in B_P(p)$, then $D(p,P)=\|p-x\|$. Again from $(ii)$ we have
	\begin{equation*}
		\|z-p\|\leq \sup_{y\in Fx}\|y-p\| \leq \|x-p\|=D(p,P). 
	\end{equation*}  
	Thus we can conclude that $\|z-p\|=D(p,P)$. Therefore $z\in B_P(p)$. Hence $Fx\subseteq  B_P(p)$ for each $x\in B_P(p)$, that is $F( B_P(p))\subseteq  B_P(p)$. Immediately it observe that $Fx$ are also closed for any $x\in B_P(p)$, on the account of $Fx\subset CL(P)$, for each $x\in P$. Consequently $F|_{ B_P(p)}: B_P(p)\to CL( B_P(p))$. Thus $F$ fulfils all the criteria of the Theorem \ref{feng_thm} on the set $B_P(p)$ and hence
	$$w^*\in \{z\in B_P(p):z\in F|_{ B_P(p)}z\}= \{z\in Z: z\in Fz\}\cap B_P(p).$$
\end{proof}

%

The succeeding example illustrates that any uniformly local contraction (\ref{mizo_local}) due to Sultana and Vetrivel \cite{mam} does not follow  Feng-Liu contraction \cite{feng} on any arbitrary metric spaces.  

\begin{example}
	Consider $Z=[0,\frac{1}{2}]\cup \{n\in \mathbb{N}: n \geq 2\}$ and $d:Z\times Z\to \mathbb{R}$ is defined by
	\begin{center}
		$d(y,z)=
		\begin{cases}
			~0 &\mbox{if}~y=z,\\
			~|y-z| &\mbox{if}~0\leq y,z\leq \frac{1}{2},\\
			~y+z &\mbox{if one of}~ y,z\notin[0,\frac{1}{2}].
		\end{cases}$
	\end{center}
	Indeed $(Z,d)$ is not a metrically convex metric space. Now define a map $F:Z\to CB(Z)\subseteq CL(Z) $ in order that 	
	\begin{center}
		$Fz=
		\begin{cases}
			~\left\{\frac{1}{2}z^2 \right\} &\mbox{if}~z\in[0,\frac{1}{2}],\\
			~\{7,2z-1\} &\mbox{if}~ z\in\{2,3,4,\cdots\}.
		\end{cases}$
	\end{center}
	Our goal is to show that $F$ is not a Feng-Liu contraction but a Mizoguchi-Takahashi local contraction. Now for $3\in Z$, $F3=\{7,5\}$. Then there are two cases that arise.
	\begin{enumerate}
		\item [(i)] Choose $7\in \{7,5\}$, then $F7=\{7,13\}$. Now we observe that $D(7,F7)=\inf\{d(7,7),d(7,13)\}=\inf\{14,20\}=14$. Again $d(3,7)=10$. Therefore $$D(7,F7)> d(3,7).$$
		
		\item[(ii)] Choose $5\in \{7,5\}$, then $F5=\{7,9\}$. Now we see that $D(5,F5)=\inf\{d(5,7),d(5,9)\}=\inf\{12,14\}=12$. Again $d(3,5)=8$. Therefore $$D(5,F5)> d(3,5).$$
	\end{enumerate} 
	Hence for $y=3\in Z$ and every $z\in F3$, we obtain that 
	$$D(z,Fz)>d(y,z).$$
	This indicates that $F$ is not a Feng-Liu contraction. On the other side, when $d(y,z)<\frac{1}{2}$, then it is easy to visualize that $y,z\in [0,\frac{1}{2}]$. Now 
	\begin{eqnarray*}
		H(Fy,Fz)&=&H\left(\left\{\frac{1}{2}y^2\right\},\left\{\frac{1}{2}z^2\right\}\right)\\
		&\leq& \frac{1}{2}(y+z)d(y,z).
	\end{eqnarray*} 
	Consider a map $k:[0,\infty)\to [0,1)$ by
	\begin{center}
		$k(s)=
		\begin{cases}
			~\frac{3}{4 } &\mbox{if}~s\in[0,1/2)\\
			~0 &\mbox{if}~ s\in[1/2,\infty).
		\end{cases}$
	\end{center}
	Therefore $k\in W$ and hence for $d(y,z)<\frac{1}{2}$, 
	$$H(Fy,Fz)\leq k(d(y,z))d(y,z).$$
\end{example}

\section*{Acknowledgment}

The first author would like to acknowledge the Ministry of Human Resource Development, India for providing financial assistance during the research work.

\end{document}